\documentclass[reqno]{alea2}
\usepackage{natbib}
\usepackage{fancyhdr}
\usepackage{graphicx}

\usepackage{amssymb}
\usepackage{amsmath}

\renewcommand{\cite}{\citet}

\makeatletter \@addtoreset{equation}{section} \makeatother

\renewcommand\theequation{\thesection.\arabic{equation}}
\renewcommand\thefigure{\thesection.\@arabic\c@figure}
\renewcommand\thetable{\thesection.\@arabic\c@table}

\theoremstyle{plain}
\newtheorem{theorem}{Theorem}[section]                                          
\newtheorem{proposition}[theorem]{Proposition}                          
\newtheorem{lemma}[theorem]{Lemma}
\newtheorem{corollary}[theorem]{Corollary}

\theoremstyle{definition}

\theoremstyle{remark}
\newtheorem{remark}[theorem]{Remark}
\newtheorem{example}[theorem]{Example}

\usepackage{color,colordvi}
\definecolor{my-blue}{rgb}{0.0,0.0,0.6}
\definecolor{my-red}{rgb}{0.5,0.0,0.0}
\definecolor{my-green}{rgb}{0.0,0.5,0.0}

\numberwithin{equation}{section}
\newcommand*{\R}{{\mathbb R}}
\newcommand*{\Z}{{\mathbb Z}}
\renewcommand*{\P}{{\mathbb P}}
\newcommand*{\E}{{\mathbb E}}
\newcommand*{\N}{{\mathbb N}}
\newcommand*{\w}{\omega}
\newcommand*{\kS}{{\mathfrak S}}
\newcommand*{\e}{\varepsilon}
\newcommand*{\cM}{{\mathcal M}}
\newcommand*{\one}{{{\rm 1\mkern-1.5mu}\!{\rm I}}}
\newcommand*{\Xtil}{{\widetilde X}}
\newcommand*{\Ybar}{{\bar Y}}
\newcommand*{\Xhat}{{\widehat X}}
\newcommand*{\Pbar}{{\bar P}}
\newcommand*{\Ebar}{{\bar E}}
\newcommand*{\nn}{\nonumber}
\newcommand*{\PMC}{{\mathrm P}}
\newcommand*{\EMC}{{\mathrm E}}

\begin{document}

\date{March 22, 2010, revised August 28, 2010, accepted December 10, 2010}

\author{Mathew Joseph}
\address{M.\ Joseph, Department of Mathematics, University of Utah,
155 South 1400 East, Salt Lake City, UT 84112-0090, USA.}
\email{joseph@math.utah.edu}
\urladdr{http://www.math.utah.edu/$\sim$joseph}
\author{Firas Rassoul-Agha}
\address{F.\ Rassoul-Agha, Department of Mathematics, University of Utah,
155 South 1400 East, Salt Lake City, UT 84112-0090, USA.}
\email{firas@math.utah.edu}
\urladdr{http://www.math.utah.edu/$\sim$firas}
\thanks{Work supported in part by NSF Grant DMS-0747758.}

\title[Almost Sure Invariance Principle for RWRE]{Almost Sure Invariance Principle for Continuous-Space Random Walk 
in Dynamic Random Environment} 

\begin{abstract}
We consider a random walk on $\R^d$ in a polynomially mixing random environment that is refreshed at each time step.
We use a martingale approach to give a necessary and sufficient condition for the almost-sure functional central limit theorem to hold.
\end{abstract}

\maketitle

\section{Introduction and main result}
\label{sec1}
Random Walk in Random Environment (RWRE) is by now a standard model of motion in disordered media. 
Put simply, a RWRE is a Markov chain on a particular space where the transition probabilities are chosen through a random experiment. In other words, we first randomly choose all the transition probabilities to get a \textit{random environment} and then we have a \textit{random walk} governed by this random environment. For the case when the walk happens on the integer lattice,  
\cite{bolt-szni-dmv}
and 
\cite{zeit-stflour} give an excellent overview.
Of course, one can use $\R^d$ in place of $\Z^d$. 
One then must account for some mixing in the environment.  In this paper, we consider the special case where the environment 
is ``refreshed'' each time step, and thus  the underlying space will in fact be $\Z\times\R^d\subset\R^{d+1}$ where $\Z$ represents time and $\R^d$ represents
space. ($\R$ is the set of real numbers, $\Z$ the integers, $\Z_+$ the nonnegative integers, and $\N$ the positive integers.)
Let us now describe our model in more detail.

The environment space is $\Omega=\big(\mathcal{M}_1(\R^d)\big)^{\Z\times\R^d}$ where $\mathcal{M}_1(\R^d)$ is the space of probability measures on $\R^d$. 
An {\sl environment} $\w$ in $\Omega$ is of the form $\w= \big(\w_{n,x}\big)_{n\in \Z, x\in \R^d}$,
where $n\in\Z$ denotes the discrete temporal coordinate and $x\in\R^d$ the continuous spatial coordinate. 
$\w_{n,x}$ is a probability measure on $\R^d$ and represents the jump probability, at time $n$, from point $x$ to a new location in $\R^d$. Denote by $\w_{n,\cdot}=\big(\w_{n,x}\big)_{x\in \R^d}$, the component of $\w$ on the time coordinate (or level) $n$.   Given an environment $\w \in \Omega$, a time $m\in \Z$, and a location $y\in \R^d$  the probability measure $P_{m,y}^{\w}$ defines a Markov chain $(Z_n)_{n\ge0}$  on $\Z \times\R^d$ as follows
\begin{align*}
&P_{m,y}^{\w}\big\{Z_0=(m,y)\big\}=1\text{ and }\\ 
&P_{m,y}^{\w}\big\{Z_{n+1}\in \{m+n+1\}\times A \,\big\vert\, Z_n=(m+n,x)\big\}= \w_{m+n,x}(A-x).
\end{align*}
$(Z_n)_{n\ge0}$ is called a random walk in environment $\w$
and $P_{m,y}^{\w}$ is called the {\sl quenched measure}.
We will be interested in random walks which start at time $0$. In this case $Z_n=(n,X_n)$ and we can just look at the evolution of $X_n$.  For simplicity of notation we will abbreviate $P_x^{\w}$ for $P_{0,x}^{\w}$. We equip $\mathcal{M}_1(\R^d)$ with the topology of weak convergence and the corresponding Borel $\sigma$-algebra.
Then equip $\Omega$ with the product $\sigma$-algebra $\kS$. 
We are given a probability measure
$\P$ on $(\Omega,\kS)$ which is stationary and ergodic under the shifts $T^{m,y}\w=(\w_{m+n,y+x})_{n\in\Z,x\in\R^d}$. $P_x=\int P_x^{\w}\,\P(d\w)$ is then called the {\sl joint measure} and its marginal on the sequence space $\big(\Z \times\R^d\big)^{\Z_+}$ is called the {\sl averaged measure} and is still denoted by $P_x$. 
Denote the expectations corresponding to $\P,P_x^{\w}, P_x$, etc, by $\E,E_x^{\w},E_x$, etc. 
Let $\kS_n$ be the $\sigma$-algebra generated by $\w_{n,\cdot}$ and let $\kS_n^+$ be
the $\sigma$-algebra generated by $(\w_{m,\cdot})_{m\ge n}$.
$C$ will denote a chameleon constant which might change value from term to term.

Before we state our assumptions let us note that the case of RWRE on $\mathbb Z^d$ is recovered from our model by letting $\w_{0,0}$ be supported on 
$\mathbb Z^d$ and then
setting $\w_{n,x}=\w_{n,[U+x]}$, where $[y]$ means we take the integral part of each of the coordinates of $y$ and $U$ is a random variable independent of $\w$
and uniformly distributed on the cube $[0,1)^d$.

We are now ready to state the assumptions on the environment measure $\P$.

\newtheorem*{A1}{\sc Assumption A1}
\begin{A1}
Time components $\big(\w_{n,\cdot}\big)_{n\in \Z}$ are i.i.d.\ under $\P$.
\end{A1}

Assumption A1 means $X_n$ is a random walk on $\R^d$ in a random environment that gets refreshed at each time step.
With A1 assumed, the law of $X_n$ under $P_0$ becomes that  of a classical random walk on $\R^d$ with jump probability $p(A)=P_0\{X_1\in A\}$. 
Thus, for example, the law of large numbers (LLN) holds if, and only if, one has $E_0[|X_1|]<\infty$. The limiting velocity then equals
\begin{equation} \label{v} v=  E_0[X_1]=\E\Big[\int\!x\, \w_{0,0}(dx)\Big].\end{equation}

We are interested in the central limit theorem (CLT), hence the next assumption.

\newtheorem*{A2}{\sc Assumption A2}
\begin{A2}
The one step jump has a finite second moment:
\begin{equation} \label{mom} E_0[|X_1|^2]=\E\Big[\int\! \vert x\vert^2\, \w_{0,0}(dx)\Big]<\infty. \end{equation}
\end{A2}

For $\e>0$ define the process 
\begin{equation}\label{bn}B_\e(t)=\sqrt{\e}(X_{[t/\e]}-[t/\e]v)\text{ for }t\ge 0.\end{equation}
Donsker's invariance principle says that the law of $B_\e$ under $P_0$ converges weakly, as $\e\to0$, to the law
of a Brownian motion with covariance matrix
\begin{equation}\label{D} \mathcal{D}= E_0[(X_1-v)(X_1-v)^T]=\E\Big[\int (x-v)(x-v)^T\, \w_{0,0}(dx)\Big].\end{equation}
Here, $A^T$ is the transpose of the matrix $A$ and a vector $a\in\R^d$ is thought of as a matrix with one column and $d$ rows.
(A Brownian motion with covariance matrix $\mathcal D$ has the same law as $\Gamma W$ with $\mathcal D=\Gamma\Gamma^T$ and $W$ a standard $d$-dimensional Brownian motion.)

We are interested in the situation where the invariance principle also holds for the laws of $B_\e$ under $P_0^\w$ for $\P$-a.e.\ $\w$. This is called the {\sl quenched
invariance principle}.
To motivate our next assumption we consider an example. Denote the local drift by
\begin{equation}\label{drift}D(\w)=E_0^\w[X_1]=\int\! x\,\w_{0,0}(dx).\end{equation}
Observe that $D(\w)$ only depends on $\w_{0,0}$.

\begin{example}\label{example}
Let $\P$ be such that $(\w_{n,0})_{n\in\Z}$ is a stationary ergodic sequence\ (valued in $\cM_1(\R^d)$) and for each $n\in\Z$ and $x\in\R^d$ 
$\w_{n,x}=\w_{n,0}$. 
Then, $E_0^\w[X_{k+1}-X_k\,|\,X_k]=D(T^{k,X_k}\w)=D(T^{k,0}\w)$. Thus, $E_0^\w[X_{k+1}-X_k]=D(T^{k,0}\w)$ and $X_n-E_0^\w[X_n]$ is 
a $P_0^\w$-martingale relative to the filtration $\sigma\{X_1,\cdot,X_n\}$. It is easy to check that the conditions for the martingale invariance principle are  satisfied;
see for example Theorem 3 of 
\cite{rass-sepp-05}.
The conclusion is that for $\P$-a.e.\ $\w$ the law of   
\begin{equation}\label{tbn} \widetilde{B}_\e(t)=\sqrt{\e}(X_{[t/\e]}-E_0^{\w}[X_{[t/\e]}]),\ t\ge 0,\end{equation}
under $P_0^\w$ converges weakly to a Brownian motion with a covariance matrix that is independent of $\w$.
On the other hand, if $(\w_{n,0})_{n\in\Z}$ is mixing enough (in particular, when A1 holds), then $E_0^\w[X_n]-nv=\sum_{k=0}^{n-1}(D(T^{k,0}\w)-v)$ satisfies its own invariance 
principle. Thus, the laws of  $(X_n-nv)/\sqrt n$ under $P_0^\w$  are not tight.
\end{example}

The above example shows that in order for the quenched invariance principle for $B_\e$ to hold one needs to assume some spatial mixing on the environment.

\newtheorem*{A3}{\sc Assumption A3}
\begin{A3}
There exists $p>26$ and a constant $C>0$ such that for all measurable $A,B\subset \mathcal{M}_1(\R^d)$, we have
\begin{equation} 
\begin{split}
&\Big\vert \P\{\w_{0,0}\in A,\w_{0,x}\in B\}-\P\{\w_{0,0}\in A\}\P\{\w_{0,x}\in B\}\Big \vert\\
&\qquad\qquad  \le \frac{C}{\vert x\vert^p}\P\{\w_{0,0}\in A\}\P\{\w_{0,x}\in B\}.
\end{split}\label{mix}\end{equation}
\end{A3}

\begin{remark}
The bound $p>26$ is established from the bounds in Proposition \ref{the proposition} below. 
It is not optimal and can be improved by more work with the same ideas.
\end{remark}

By a standard approximation argument it follows from A3 that if $f$ and $h$ are two 
nonnegative functions in $L^2(\P)$ that are $\sigma\{\w_{0,0}\}$-measurable, then 
 \begin{equation} 
 \label{mix2}
 \big\vert \E[f(\w)h(T^{0,x}\w)]-\E[f]\E[h]\big\vert \le\frac{C}{|x|^p} \E[f]\E[h]. \end{equation}

Our last assumption concerns the regularity of the environment. Let $\delta_z$ denote the pointmass at $z$.

\newtheorem*{A4}{\sc Assumption A4}
\begin{A4}
$\P$ satisfies the following:
\begin{equation}  \P\{\exists z:\w_{0,0}=\delta_z\}<1.\label{regular}\end{equation}
\end{A4}

Say A4 fails to hold. Let \[A=\{\mu\in \cM_1(\R^d):\exists z=z(\mu)\in\R^d\text{ such that }\mu=\delta_z\}.\] 
If $\alpha_{n,x}$ denotes the marginal of $\P$ on $\sigma(\w_{n,x})$, then $\alpha_{n,x}(A)=1$ for any fixed $n\in\Z$ and $x\in\R$.
By independence of $\w_{1,\cdot}$ and $\kS_0$ and the disintegration lemma we have
\begin{align*}
\P\big\{\exists(z_1,z_2):\w_{0,0}=\delta_{z_1},\w_{1,z_1}=z_2\big\}&=\int\!\one_A(\mu)\,\alpha_{1,z(\mu)}(A)\,\alpha_{0,0}(d\mu)\\
&=\alpha_{0,0}(A)=1.
\end{align*}
This implies that given the environment $\w$, the walk $(X_n)_{n\ge0}$ is nonrandom under $P_0^\w$. In this case, there are no fluctuations in the quenched
walk and the invariance principle fails to hold unless $\w_{0,0}$ is also nonrandom under $\P$, in which case the invariance principle is degenerate with a 
vanishing covariance matrix. 

\begin{remark}
It is noteworthy that when A4 fails to hold the situation, even though degenerate, is similar to the one in Example \ref{example}.
Indeed, $X_n-E_0^\w[X_n]=0$ and thus a degenerate quenched invariance principle holds for $\widetilde B_\e\equiv0$. Moreover, 
$E_0^\w[X_n]-nv=X_n-nv$ and thus an invariance principle holds for the processes $\{\sqrt{\e}(E_0^\w[X_{[t/\e]}]-[t/\e]v):t\ge0\}$. 
\end{remark}

We can now formulate the main theorem of this paper.

\begin{theorem} \label{thm}Assume the environment measure $\P$ is shift invariant and  satisfies the independence assumption A1 and the mixing assumption A3.  Then a quenched invariance principle holds if and only if the moment assumption A2 and the regularity assumption A4 are satisfied. That is, for $\P$-a.e.\ $\w$ the distribution of $B_\e$ induced by $P_0^\w$ converges weakly to the distribution of a Brownian motion with covariance matrix $\mathcal{D}$ given by \eqref{D}. Moreover, $n^{-1/2}\max_{k\le n}\vert E_0^{\w}[X_k]-kv \vert $ converges to $0$ $\P$-a.s.\ and the same invariance principle holds for the distribution of $\widetilde B_\e$ induced by $P_0^\w$.
\end{theorem} 

There are three major approaches that have  been used to prove quenched central limit theorems for RWRE, two of which were directly used to deal with the special case of the above theorem where $\w_{0,0}$ is $\P$-almost-surely supported on $\Z^d$ and $\{\w_{n,x}:n\in\Z,x\in\Z^d\}$ is an i.i.d.\ sequence; the so-called {\sl random walk in space-time product random environment}. 

One approach is via Fourier-analytic methods; see  
\cite{bold-minl-pell-04}. 
This approach requires exponential moment controls on the step of the random walk, uniformly in the environment $\w$; i.e.\ that $\sup_\w E_0^\w[e^{\lambda|X_1|}]<\infty$ for some $\lambda>0$. 
Recently, the above authors showed that their method can handle spatial mixing in the environment and proved a weaker version of Theorem \ref{thm}. Namely, \cite{bold-minl-pell-09} 
assume exponential spatial mixing (rather than polynomial, as in assumption A3), that transitions $\w_{n,x}$ have a density relative to the Lebesgue measure on $\R^d$ and, most restrictive, the assumption that transition measures $\w_{n,x}$ are small random perturbations of a nonrandom jump measure $p(y)\,dy$. 

On the other hand, 
\cite{rass-sepp-05}
consider the Markov chain $(T^{n,X_n}\w)$ of the \textit{environment as seen by the particle} and use 
\label{MC arguments}%
general Markov chains arguments 
(introduced by \cite{kipn-vara-86} for reversible Markov chains then generalized by \cite{maxw-wood-00}, \cite{rass-sepp-08},
and \cite{derr-lin-03})  
to prove  Theorem \ref{thm} for random walk in space-time product random environment.  One advantage of this approach is that it can be made to work in more general
RWRE settings; see \cite{rass-sepp-06,rass-sepp-07,rass-sepp-09}. 
The main step in this approach is a subdiffusive bound on the variance of the quenched mean (see Theorem \ref{thm2}).
The goal of the present paper is to show that this approach is quite natural and works even when (mild) spatial mixing is present in the environment, still giving a necessary and sufficient
condition for the quenched invariance principle to hold. It is noteworthy that \cite{dolg-live-09} use a similar method to prove a quenched  invariance principle 
in the case when $\w_{n,\cdot}$ forms a Gibbsian Markov chain, generalizing the independence assumption A1 but strengthening the mixing assumption A3. 

The third approach, used by \cite{berg-zeit-08}, is based on a concentration inequality (Lemma 4.1 of \cite{bolt-szni-02}) 
that shows that the quenched process is not ``too far'' from the averaged one and then appeals
to the averaged central limit theorem. Even though this has not been applied directly to the space-time case, there is no reason why it would not succeed in
providing an alternate proof of Theorem \ref{thm}. Incidentally, to prove the concentration inequality one needs the same variance bound as in the aforementioned martingale approach.

We end this introduction with the main tool in the proof of the quenched invariance principle.
For a $f\in L^1(\Omega,\P)$ define
\[ \Pi f(\w)=\int f\big(T^{1,x}\w\big)\, \w_{0,0}(dx) .\]
The operator $\Pi-I$ defines the generator of the Markov chain of the environment as seen from the particle. This is the process on $\Omega$ with transitions
\[ \tilde\pi(\w,A)=P_0^{\w}\{T^{1,X_1}\w \in A\}.\]

\begin{theorem}\label{thm2}
Let $\P_\infty\in\cM_1(\Omega)$ be stationary ergodic for the Markov chain with generator $\Pi-I$. Let $\E_\infty$ denote the corresponding expectation.
Assume $\int E_0^\w[|X_1|^2]\,\P_\infty(d\w)<\infty$. Assume  there exists an $\eta\in(0,1)$ such that 
	\begin{align}\label{bound}
	\E_\infty[|E_0^\w[X_n]-\E_\infty[E_0^\w[X_n]]|^2]={\mathcal O}(n^\eta).
	\end{align}
Then, $n^{-1/2}\max_{k\le n}\vert E_0^{\w}[X_k]-kv \vert $ converges to $0$ $\P_\infty$-a.s.\ and for $\P_\infty$-a.e.\ $\w$ both the law of $B_\e$  and that
of $\widetilde B_\e$ under $P_0^\w$ converge weakly to {\rm(}the same{\rm)} Brownian motion with a nonrandom
covariance matrix.  
\end{theorem}

\begin{proof}
When $\w_{0,0}$ is supported on $\Z^d$ this theorem is a special case of Theorem 2 of \cite{rass-sepp-05}. However, the proof 
goes through word for word when $\Z^d$ is replaced by $\R^d$. The main idea is to first observe that $M_n=X_n-\sum_{k=0}^{n-1}D(T^{k,X_k}\w)$ is a
martingale. Next, one uses the Markov chain arguments alluded to on page \pageref{MC arguments} to
decompose $\sum_{k=0}^{n-1}D(T^{k,X_k}\w)=\bar M_n+R_n$ with $\bar M_n$
another martingale and, due to \eqref{bound}, $R_n=o(\sqrt n)$, $P_0^\w$-almost surely for $\P_\infty$-almost every $\w$. 
This is where the hard work is.
The result then follows from the invariance principle for stationary ergodic martingales. 
\end{proof}

In Section \ref{pinf} we construct a probability measure $\P_{\infty}$ which is invariant and ergodic for the \textit{environment Markov chain}. We also compare $\P_{\infty}$ to $\P$. In Section \ref{secbound} we check condition \eqref{bound} and prove Theorem \ref{thm}.

\section{Construction of the invariant measure}
\label{pinf}

Let us start with some notation. Denote the quenched law of $X_n$ by $\pi_x^{\w,n}(A)=P_x^{\w}\{X_n\in A\}$. This is a  probability measure on $\R^d$. Also let $\pi_x^n(A)=P_x\{X_n\in A\}=\int\pi_x^{\w,n}(A)\P(d\w)$.  $\P_n$ will denote the probability measure on $\Omega$ defined as
\[ \P_n(S) =P_0\{T^{n,X_n}\w\in S\}=\iint \one\{T^{n,y}\w \in S\}\, \pi_0^{\w,n}(dy)\P(d\w).\]
This is the law of the environment as seen from $X_n$.
Note that for any bounded function $f(\w)$,
	\[\int f(\w)\,\P_n(d\w)=\iint f\big(T^{n,y}\w\big)\,\pi_0^{\w,n}(dy)\P(d\w).\]

The rest of the section is devoted to the proof of the following theorem.

\begin{theorem}\label{Pinfty}
Let $\P$ be shift invariant and satisfy the independence assumption A1. Then, there exists a probability measure $\P_\infty$ on $(\Omega,\kS)$ such that  
$\P_{\infty}\big\vert_{\kS^+_{-n}}=\P_n\big\vert_{\kS^+_{-n}}$ for any $n\ge0$. Moreover, $\P_\infty$ is invariant and ergodic for the Markov chain with transition operator $\Pi$.
\end{theorem}

\begin{proof}
We first show that measures $\P_n$ form a consistent family. 

\begin{lemma}
For $n\ge m\ge0$ we have 
$\P_n  \big\vert_{\kS_{-m}^+}=\P_m \big\vert_{\kS_{-m}^+}.$
\end{lemma}
\begin{proof}
Fix an $A \in \kS_{-m}^+$. Use the Markov property to write
\begin{align*}
\P_n(A)&=  \iint  \one\{T^{n,x} \w \in A\}\,\pi_0^{\w,n}( dx) \P(d\w)\\
&= \iiint\one\{T^{m,x} T^{n-m,0}\w \in A\}\,\pi_0^{\w,n-m}( dy)\pi_y^{T^{n-m,0}\w,m}( dx) \P(d\w).
\end{align*}
By shift invariance and the independence assumption A1
\begin{align*}
\P_n(A)&=\iiint \one\{T^{m,x} \w \in A\} \,\pi_0^{T^{-(n-m),0}\w,n-m}( dy)\pi_y^{\w,m}( dx) \P(d\w)\\
&=\int  \Big[\iint \one\{T^{m,x} \w \in A\}\, \pi_y^{\w,m}( dx)\P(d\w)\Big]\,\pi_0^{n-m}(dy).
\end{align*}
Using shift invariance again
\begin{align*}
&\iint \one\{T^{m,x} \w \in A\} \,\pi_y^{\w,m}( dx)\P(d\w)\\
&\qquad\qquad=\iint \one\{T^{m,x} \w \in A\}\, \pi_0^{T^{0,y}\w,m}(-y+dx)\P(d\w) \\
&\qquad\qquad= \iint \one\{T^{m,x-y} \w \in A\} \,\pi_0^{\w,m}(-y+dx)\P(d\w)\\
&\qquad\qquad= \iint \one\{T^{m,z} \w \in A\} \,\pi_0^{\w,m}( dz)\P(d\w)=\P_m(A) .
\end{align*}
We have thus shown that for $A \in \kS_{-m}^+$ we have  $\P_n(A)=\P_m(A)$.
\end{proof}

Kolmogorov's consistency theorem now gives the existence of a probability measure $\P_{\infty}$ such that 
$\P_{\infty}\big\vert_{\kS_{-n}^+}=\P_n\big\vert_{\kS_{-n}^+}$ for all $n\ge0$. 

Recall the transition operator  $\Pi$ of  the Markov chain of the environment as seen from the point of view of the particle $(T^{n, X_n}\w)$. 
We now prove that $\P_\infty$ is invariant and ergodic for this Markov chain.

\begin{lemma}
Probability measure $\P_{\infty}$ is invariant under $\Pi$.
\end{lemma}

\begin{proof}
Let $f$ be  a bounded  $\kS_{-k}^+$-measurable function. Then $\Pi f$ is also bounded and $\kS_{-k+1}^+$-measurable. Now write
\[\int \Pi f(\w)\,\P_n(d\w)= \iint f(T^{1,x}T^{n,y}\w)\,\pi_0^{T^{n,y}\w,1}(dx)\pi_0^{\w,n}(dy)\P(d\w).\]
Make the change of variables $(x,y)$ to $(z,y)$ where $z=x+y$ and note that $\pi_0^{T^{n,y}\w,1}(B)=P_{n,y}^{\w}\{X_1\in y+B\}$, then
use the Markov property to conclude that
\begin{align*}
\int \Pi f(\w)\,\P_n(d\w)&= \iint  f(T^{n+1,z}\w)\,\pi_0^{\w,n+1}(dz)\P(d\w)\\
&=\int f(\w)\,\P_{n+1}(d\w). \end{align*}
Taking $n\ge k$ shows that $\int \Pi f\,d\P_\infty=\int f\,d\P_\infty$.
\end{proof}

\begin{lemma}
The invariant measure $\P_{\infty}$ is ergodic for the Markov chain with generator $\Pi-I$.
\end{lemma}
\begin{proof}
The proof is identical to that of Lemma 1 of \cite{rass-sepp-05} and is omitted. Roughly, the idea is that since moves of the Markov chain consist of shifts,
absorbing sets are shift-invariant and thus of trivial $\P$-measure. The claim then follows from approximating with local sets and 
using equality of the restrictions of $\P$ and $\P_\infty$ onto $\kS_{-n}^+$. 
\end{proof}

The proof of Theorem \ref{Pinfty} is complete. 
\end{proof}

\section{Bound on the variance of the quenched mean}
\label{secbound}

We now have a probability measure $\P_{\infty}$ on $(\Omega, \kS)$ that is invariant under $\Pi$ and ergodic for the Markov chain on $\Omega$ with generator $\Pi-I$. The next important step is to verify that it satisfies \eqref{bound} of Theorem \ref{thm2}, i.e.\ that the variance of the quenched mean $E_0^\w[X_n]$ is
subdiffusive. 

\begin{proposition}\label{the proposition}
Assume $\P$ is shift invariant and satisfies A1 through A4. Then, $\P_\infty$ from Theorem \ref{Pinfty} satisfies \eqref{bound} with $\eta\le1/2+13/p$, where
$p$ is the exponent in Assumption A3. 
\end{proposition}

\begin{proof}
Since $\P_{\infty}\big\vert_{\kS_0^+}=\P_0\big\vert_{\kS_0^+}=\P\big\vert_{\kS_0^+}$ and the quantity inside the $\E_{\infty}$ expectation in \eqref{bound}  is measurable with respect to $\kS_0^+$, \eqref{bound} can be rewritten as 
\begin{equation} 
\label{bound2}
\E\Big[\big\vert E_0^{\w}[X_n]-nv\big\vert^2\Big] 
={\mathcal O}(n^{\eta}).
\end{equation}
Define $g(\w)=E_0^\w[X_1]-v$. Note that $\E[g]=0$. A simple computation gives 
\begin{equation}
\label{bound3}
\E\Big[\big\vert E_0^{\w}[X_n]-nv\big\vert^2\Big] = \sum_{k,\ell=0}^{n-1}\iiint g(T^{k,x}\w)\cdot g(T^{\ell,y} \w)\,\pi_0^{\w,k}(dx)\pi_0^{\w,\ell}(dy)  \P(d\w)
\end{equation}
By the moment assumption A2, the fact that $g$ is $\kS_0$-measurable, and the $\Pi$-invariance of $\P_\infty$, we have
	\begin{align*}
	\iint |g(T^{k,x}\w)|^2\,\pi_0^{\w,k}(dx)\P(d\w)&=\iint |g(T^{k,x}\w)|^2\,\pi_0^{\w,k}(dx)\P_\infty(d\w)\\
	&=\E_\infty[|g|^2]=\E[|g|^2]<\infty.
	\end{align*}
Using the inequality $|a\cdot b|\le2(|a|^2+|b|^2)$ we see that
	\[\iiint \Big\vert g(T^{k,x}\w)\cdot g(T^{\ell,y} \w)\Big\vert\,\pi_0^{\w,k}(dx)\pi_0^{\w,\ell}(dy)  \P(d\w)<\infty.\]

Consider a term in the sum in \eqref{bound3} with $k<\ell$. Since $\int  g(T^{k,x}\w)\,\pi_0^{\w,k}(dx)$ and $\pi_0^{\w,\ell}$  are measurable with respect to $\sigma\{\w_{m,\cdot}:m\le \ell-1\}$ and $g(T^{\ell,y}\w)$ is $\kS_\ell$-measurable, we have by Fubini's theorem
\begin{align*}
&\iiint  g(T^{k,x}\w)\cdot g(T^{\ell,y} \w)\,\pi_0^{\w,k}(dx)\pi_0^{\w,\ell}(dy)  \P(d\w)\\
&\qquad= \int\Big[\int g(T^{\ell,y}\w)\P(d\w)\Big] \cdot\Big[  \iint g(T^{k,x}\tilde\w)\,\pi_0^{\tilde\w,k}(dx)\pi_0^{\tilde\w,\ell}(dy) \P(d\tilde\w)\Big]=0.
\end{align*}
This shows that terms in \eqref{bound3} with $k\ne\ell$ vanish. Thus,
\[ \E\Big[\big\vert E_0^{\w}[X_n]-nv\big\vert^2\Big] = \sum_{k=0}^{n-1}\iiint g(T^{k,x}\w)\cdot g(T^{k,y} \w)\,\pi_0^{\w,k}(dx)\pi_0^{\w,k}(dy)  \P(d\w).\]

For two independent random walks $X_k$ and $\Xtil_k$ in the same random environment $\w$, define
\[ \pi_{x,y}^{\w,k}( A,  B)=P_{x}^{\w}\{X_k \in A\}P_y^{\w} \{\Xtil_k \in B\}\]
and 
\[  \pi_{x,y}^{k}( A,  B)=\E\Big[P_{x}^{\w}\{X_k \in A\}P_y^{\w} \{\Xtil_k \in B\}\Big].\]

Since $\pi_0^{\w,k}$ is $\sigma\{\w_{m,\cdot}:m\le k-1\}$-measurable and $g(T^{k,x}\w)$ and $g(T^{k,y}\w)$ are $\kS_k$-measurable, another application of  Fubini's theorem shows that 
\begin{equation}
\label{bound4}
 \E\Big[\big\vert E_0^{\w}[X_n]-nv\big\vert^2\Big] = \sum_{k=0}^{n-1}\iint \phi(y-x)\, \pi_{0,0}^k(dx,dy),
  \end{equation}
   where $\phi(x)=\E\big[g(\w)\cdot g(T^{0,x}\w)\big]$. 
   
Consider now the Markov chain $Y_k$ with transition probabilities given by
  \[ P\{Y_1 \in A \,|\, Y_0=x\} =\iint \one\{z-y\in A\}\,\pi_{0,x}^1(dy,dz).\]
We will use $\PMC_x$ to denote the law of this Markov chain, when started at $Y_0=x$. $\EMC_x$ will denote the corresponding expectation.
Due to the independence assumption A1, the law of $\Xtil_k-X_k$, induced by $\int P_0^\w\otimes P_x^\w\,\P(d\w)$, is the same as that of $Y_k$, given $Y_0=x$.
 So \eqref{bound4} now becomes
  \begin{equation} \label{eq1}\E\big[\vert E_0^{\w}[X_n]-nv\vert^2\big]=\sum_{k=0}^{n-1}\EMC_0\big[\phi(Y_k)\big].\end{equation}  
  
To bound  the right-hand side of \eqref{eq1} we start with a bound on the function $\phi$. 

\begin{lemma}
\label{cor}
Let $\P$ be shift invariant and satisfy the moment assumption A2 and the mixing assumption A3.
Then, there exists a constant $C>0$ so that  $\vert \phi(x)\vert \le \frac{C}{\vert x\vert^p}$,
where $p$ is as in \eqref{mix}.
\end{lemma}
\begin{proof}
Write $g=g^+-g^-$ where $g^+$ and $g^-$ are the positive and negative parts of $g$, coordinate by coordinate. 
Then
\begin{align}
\label{phi}
\begin{split}
|\phi(x)| =\Big \vert \E \Big[g^+(\w)\cdot g^+(T^{0,x}\w)&-g^-(\w)\cdot g^+(T^{0,x}\w)\\
&-g^+(\w)\cdot g^-(T^{0,x}\w)+g^-(\w)\cdot g^-(T^{0,x}\w)\Big]\Big\vert.
\end{split}
\end{align}
From \eqref{mix2} we have
\[ \big(1-\tfrac{C}{\vert x\vert^p}\big) |\E[g^+]|^2 \le \E\big[g^+(\w)\cdot g^+(T^{0,x}\w)\big]\le\big(1+\tfrac{C}{\vert x\vert^p}\big) |\E[g^+]|^2 \]
and 
\[ \big(1-\tfrac{C}{\vert x\vert^p}\big) \E[g^+]\cdot\E[g^-] \le \E\big[g^-(\w)\cdot g^+(T^{0,x}\w)\big]\le\big(1+\tfrac{C}{\vert x\vert^p}\big) \E[g^+]\cdot\E[g^-] .\]
Observing that $\E[g^+]=\E[g^-]$ and subtracting the above two expressions we have
\[ \Big\vert \E\big[ g^+(\w)g^+(T^{0,x}\w)-g^-(\w)g^+(T^{0,x}\w) \big]\Big\vert \le \frac{C}{\vert x\vert^p}.\]
 A similar bound can be obtained for the last two terms in \eqref{phi}.
\end{proof}

Now return to \eqref{eq1}.  For simplicity of notation define the measure
\[ \nu_0^k(A)=\PMC_0\{Y_k \in A\}.\]
Fix $\varepsilon>0 $ so that $p\varepsilon\le 1$, where $p$ is the exponent from \eqref{mix}. We get 
\[ \sum_{k=0}^{n-1}\EMC_0\big[\phi(Y_k)\big] =\sum_{k=0}^{n-1}\int_{\vert y\vert>n^{\varepsilon}} \phi(y)\,\nu_0^k(dy)+\sum_{k=0}^{n-1}\int_{\vert y\vert \le n^{\varepsilon}} \phi(y)\,\nu_0^k(dy).\]
By Lemma \ref{cor} the first term is bounded by
\[\sum_{k=0}^{n-1}\int_{\vert y\vert>n^{\varepsilon}} \vert\phi(y)\vert\,\nu_0^k(dy)\le \sum_{k=0}^{n-1} \int_{\vert y\vert\ge n^{\varepsilon}}\tfrac{C}{\vert y\vert^p}\,\nu_0^k(dy)\le Cn^{1-p\varepsilon}.\]
 Since $|\phi(y)|\le\E[|g|^2]$, \eqref{bound2} would follow if we show 
\begin{equation} \label{eq2} \sum_{k=0}^{n-1} \PMC_0\{Y_k \in [-n^{\varepsilon},n^{\varepsilon}]^d\} \le Cn^{\eta'}\end{equation}
for some $\eta'<1$. To this end, we will need to compare the Markov chain $Y_k$ to a random walk $\Ybar_k$ whose transition probabilities are given by 
\begin{align*}
P\{\Ybar_1 \in A \,|\, \Ybar_0=x\}&=\iint \one\{z-y\in A\}\,\pi_{0}^1(dy)\pi_x^1(dz)\\
&=\iint \one\{\tilde z-y\in A-x\}\,\pi_{0}^1(dy)\pi_0^1(d\tilde z)\\
&=P\{\Ybar_1\in A-x\,|\,\Ybar_0=0\}.
\end{align*}
(Recall the definition of $\pi_0^1$ introduced in Section \ref{pinf}.)
While $Y_k=\Xtil_k-X_k$ where $X$ and $\Xtil$ are independent walks in the same environment, $\Ybar_k=\Xhat_k-X_k$ where $X$ and $\Xhat$ are independent walks in independent environments. We will use $\Pbar_y$ and $\Ebar_y$ for the law and expectation of the $\Ybar$ walk starting at $y$.
We will also denote by $P_{x,y}^{\w}$ the joint law of two independent random walks $X$ and $\Xtil$ in the same environment $\w$ starting at $X_0=x$ and $\Xtil_0=y$.

To prove \eqref{eq2} we will adapt the strategy in Appendix A of \cite{rass-sepp-09} to our situation. 
We first show that the Markov chain expected exit time from boxes grows at most exponentially in the diameter of the box. Then, using a multiscale
recursion argument, we improve this to a polynomial bound. The upshot is that the Markov chain does not spend a lot of time in $[-n,n]^d$. On the other hand, 
outside this box the chain is close to the symmetric random walk which has long excursions. 

We start with a few crucial observations about the Markov chains $Y$ and $\Ybar$,  which we put in a lemma. Let $a^j$ denote the $j$-th coordinate of a vector $a$.
Let $\lceil s\rceil$ denote the smallest integer larger than $s$.

\begin{lemma}
The following statements hold.
\begin{itemize}
\item[{\rm(a)}] $Y_1$ and $\Ybar_1$ have a finite second moment and are symmetric about $0$.
\item[{\rm(b)}]  If $\PMC_0\{Y_1^j\ne0\}>0$, then there exist $M>0$, $L>0$, and $\delta>0$ such that we have 
\begin{align}\label{ellipt2}
\PMC_x\big\{Y_1^j-x^j\ge L/(\lceil\tfrac M{x^j}\rceil\vee1)\big\}\ge {\delta^2}/{4}(\lceil\tfrac M{x^j}\rceil\vee1)^2\text{\quad if }x^j>0.
\end{align}
\item[{\rm(c)}] If $\PMC_0\{Y_1^j\ne0\}>0$,  then there exists a $\kappa\in(0,M)$ such that 
	\begin{align}\label{ellipt3}
	\PMC_x\big\{|Y_1^j|>\kappa\big\}>\kappa\text{\quad if }|x^j|\le\kappa.
	\end{align}
\item[{\rm(d)}] Let $U_r=\inf\{ n\ge 0: Y_n \notin [-r,r]^d\}$ be the exit time from the centered cube of side length $2r$ for the Markov chain $Y$. Then  there is a constant $0<K< \infty$ such that 
\begin{align}\label{exponential growth}
\sup_{x \in [-r,r]^d}\EMC_x[U_r] \le K^r \text{ for all } r\ge 0.
\end{align}
\end{itemize}
\end{lemma}
\begin{proof}
The second moments are finite because of \eqref{mom}. Exchanging the roles of $X$ and $\Xtil$ (respectively, $X$ and $\Xhat$), it is clear that both $Y_1$ and $\Ybar_1$ are symmetric about $0$. We next prove (b). 

Use mixing \eqref{mix} and translation invariance in the second line below to write for all $a$ and $L>0$
\begin{align*}
\PMC_x\{Y_1^j-x^j\ge L\}&\ge \E\big[P_0^\w\{X_1^j\le a\}P_x^\w\{X_1^j\ge a+x^j+L\}\big]\\
&\ge P_0\{X_1^j\le a\}P_0\{X_1^j\ge a+L\}-\tfrac{C}{|x|^p}.
\end{align*}
Since $X_1^j$ is not deterministic one can choose $a$ and $L>0$ so that the first term in the second line above is a positive number $2\delta\le1$.
Let $M$ be such that the second term is less than $\delta$ when $x^j\ge M$. 
The upshot is that \eqref{ellipt2} holds for $x^j\ge M$: 
\begin{align}\label{ellipt1}
\PMC_x\{Y_1^j-x^j\ge L\}\ge\delta\ge\tfrac{\delta^2}{4}.
\end{align}
Assume there exists an $x$ such that $0<x^j<M$ and
\[\PMC_x\big\{Y_1^j-x^j\ge \tfrac Ln\big\}< \tfrac{\delta^2}{4n^2},\]
where $n=\lceil M/x_j\rceil\ge1$.
Then, by Chebyshev's inequality
\[\P\Big[\w:P_{0,x}^\w\big\{Y_1^j-x^j\ge \tfrac{L}n\big\}\ge \tfrac{\delta}{2n}\Big]<\frac{\delta}{2n}\,.\]
By shift invariance we have
\begin{align}\label{continue}
\P\Big[\w:P_{ix,(i+1)x}^\w\big\{Y_1^j-x^j\ge \tfrac{L}{n}\big\}\ge \tfrac{\delta}{2n}\text{ for some }i\in[0,n)\Big]<\frac{\delta}2.
\end{align}
Now observe that
\[P_{0,nx}^\w\big\{Y_1^j-nx^j\ge L\big\}\le\sum_{i=0}^{n-1}
P_{ix,(i+1)x}^\w\big\{Y_1^j-x^j\ge \tfrac{L}{n}\big\}.\]
To see this consider independent variables $X_1^{(i)}$ with law $P_{ix}^\w$ and note that 
\[X_1^{(n)}-X_1^{(0)}-nx=\sum_{i=0}^{n-1}(X_1^{(i+1)}-X_1^{(i)}-x).\]
Picking up from \eqref{continue} and letting $y=nx$ we have 
\[\P\Big[\w:P_{0,y}^\w\big\{Y_1^j-y^j\ge L\big\}\ge\tfrac{\delta}{2}\Big]<\frac{\delta}2\]
which in turn implies
\[\PMC_y\big\{Y_1^j-y^j\ge L\big\}<\tfrac{\delta}{2}+\tfrac{\delta}2.\]
Since $y^j\ge M$, this contradicts \eqref{ellipt1}. Part (b) is proved. Next, we prove (c).

Assume \eqref{ellipt3} is false. Then, for each $n\ge1$ there exists $x_n$ such that $|x_n^j|\le1/2^n$ and $\PMC_{x_n}\{|Y_1^j|>4^{-n}\}\le 4^{-n}.$
Since $\PMC_0\{Y_1^j\ne0\}>0$, there exist numbers $a<b$ and $\alpha>0$ such that the set 
	\[A=\Big\{\w:P_0^\w\{X_1^j\le a\}\ge\alpha\text{ and }P_0^\w\{X_1^j\ge b\}\ge\alpha\Big\}\]
has positive $\P$-measure. 
Consider now the sets \[D_n=\Big\{\w:P^\w_{0,x_n}\{|Y_1^j|>4^{-n}\}>2^{-n}\Big\}.\]
By Chebyshev's inequality $\P(D_n)\le 2^{-n}$ and by Borel-Cantelli's lemma $D_n$ occur finitely often, $\P$-almost surely.  But for $\w\in D_n^c\cap A$ we have
\begin{align*}
2^{-n}&\ge P_{0,x_n}^\w\{|X_1^j-\Xtil_1^j|>4^{-n}\}\\
&\ge P_0^\w\{X_1^j\le a\}P_{x_n}^\w\{\Xtil_1^j>a+4^{-n}\}\\
&\ge\alpha P_{x_n}^\w\{\Xtil^j>a+4^{-n}\}.
\end{align*}
The same holds for $P_{x_n}^\w\{\Xtil_1^j\ge b-4^{-n}\}$. 
This implies that, with positive probability, both $P_{x_n}^\w\{\Xtil_1^j\le a+4^{-n}\}$ and $P_{x_n}^\w\{\Xtil_1^j\ge b-4^{-n}\}$ converge to $1$ as $n\to\infty$. 
This is a contradiction, since the two add up to less than one for $n$ large. \eqref{ellipt3} is proved and it still holds if one takes
a smaller $\kappa>0$ to ensure $\kappa<M$. (c) is proved. 

To prove (d) observe that by assumption A4 there exists a $j$ such that $\PMC_0\{Y_1^j\ne0\}>0$. By (b) and (c) the probability that the $Y$ Markov chain exits the cube $[-r,r]^d$ 
in fewer than $1+2rM/(L\kappa)$ steps is at least $\kappa[(\delta^2\kappa^2)/(16M^2)]^{2rM/(L\kappa)}$. The exit time $U_r$ is thus stochastically dominated by 
$1+2rM/(L\kappa)$ times a geometric random variable with mean $[16M^2/(\delta^2\kappa^2)]^{2rM/(L\kappa)}/\kappa$.
\end{proof}

Let $B_r=[-r,r]^d$.

\begin{proposition} 
 \label{a3}
 Let $r_0=r^{\frac{3}{p}}$. There exist constants $0<\alpha_1, A_1<\infty$ such that for large enough $r$
 \[ \inf_{y \in B_r\backslash B_{r_0} }\PMC_y\big\{\text{without entering $B_{r_0}$  chain $Y$ exits $B_r$ by time $A_1r^3$}\big\}\ge \frac{\alpha_1}{r}.\]
 \end{proposition}
  \begin{proof}
 Let us call $E$ the event inside the braces in the statement above and recall that $U_r$ is the exit time of the Markov chain $Y$ from the box $B_r$.  We have
 \begin{align*}
  &\PMC_y(E)= \sum_{1\le k \le A_1r^3}\PMC_y( E, U_r=k)\\
  &= \sum_{1\le k\le A_1r^3}\! \int P_{0,y}^{\w}\Big\{r_0<\vert X_i-\Xtil_i \vert \le r  \text{ for }i<k\text{ and }\vert X_k-\Xtil_k\vert>r\Big\}\,\P(d\w).
     \end{align*}
 
 The right-hand side of the above expression equals
 \[ \sum_{1\le k\le A_1r^3}\int \cdots \int  \one\{\vert x_k-\tilde{x}_k\vert>r\}\pi_{x_{k-1}}^{T^{k-1,0}\w,1}(d x_k)\pi_{\tilde{x}_{k-1}}^{T^{k-1,0}\w,1}(d \tilde{x}_k)\cdots \qquad \qquad\]
 \[   \cdots\Big[\prod_{i=1}^{k-1} \one\{r_0<\vert x_i-\tilde{x}_i\vert\le r\}\pi_{x_{i-1}}^{T^{i-1,0}\w,1}(d x_i)\pi_{\tilde{x}_{i-1}}^{T^{i-1,0}\w,1}(d \tilde{x}_i)\Big]\P(d\w)\] 
Here, we have taken $x_0=0$ and $\tilde{x}_0=y\in B_r\setminus B_{r_0}$. Note that $\pi_{x_{i-1}}^{T^{i-1,0}\w,1}(d x_i)$ is a measure which depends only on the part of the environment $\{\w_{i-1,\cdot}\}$. Because the different environment levels $\w_{i,\cdot}$ are independent  under $\P$ and since we have the spatial mixing condition \eqref{mix}, we have
 \[ \PMC_y(E) \ge \big(1-\tfrac{C}{r_0^p}\big)^{A_1r^3} \Pbar_y(E).\]
By our assumption on $r_0$ we have that $\PMC_y(E)\ge C\Pbar_y(E)$. So now we just need to bound $\Pbar_y(E)$ from below. By assumption A4 the $\Ybar$ walk is nondegenerate along some  direction. Let the $j$-th coordinate $\Ybar^j$ be nondegenerate. Let $\zeta$ denote the time  $\Ybar^j$ exits $(r_0,r]$. Also denote by $\Pbar^j_{y^j}$ the law of the random walk $\Ybar^j$ starting at $y^j$.  Then 
 \begin{equation}
 \label{eq4}
 \Pbar_y(E^c)\le  \Pbar^j_{y^j}\{\Ybar^j_\zeta \le r_0\}+\Pbar^j_{y^j}\{\zeta>A_1r^3\}.
 \end{equation}
  
Time $\zeta$ is bounded above by the exit time $U^j_r$ from interval $[-r,r]$. 
It follows from Theorem 1 of \cite{prui-81} that $\Ebar^j_{y^{j}}[\zeta]\le Cr^2$. Thus the second term on the right-hand side of \eqref{eq4} is bounded by $C/(A_1r)$.

Observe next that since $\Ybar^j$ is symmetric, starting at $0$ it has a chance of $1/2$ to exit $(-1,1)$ into $[1,\infty)$. Say it exits at point $z_1\ge1$. Then, 
it has a chance of $1/2$ to exit $(-1,2z_1+1)$ into $[2z_1,\infty)$. Repeating this scenario, the walk can exit $(-1,r]$ into $(r,\infty)$ in at most $\log(r+1)/\log2$ steps.
This shows that if $\xi(r)$ is the exit time of the walk $\Ybar^j$ from $(-1,r]$, then
	\begin{align*}
	\Pbar_0\{\Ybar_{\xi(r)}>r\}\ge \tfrac 1{r+1}\,.
	\end{align*}

Using a coupling argument one sees that $\Pbar_{y^j}\{\Ybar^j_\zeta>r\}$ increases with $y^j$. Hence, for $r_0+1\le y^j\le r$ we have
	\[\Pbar_{y^j}\big\{\Ybar_\zeta^j>r\big\}\ge\Pbar_{r_0+1}\big\{\Ybar_\zeta^j>r\big\}=\Pbar_0\big\{\Ybar_{\xi(r-r_0-1)}^j>r-r_0-1\big\}\ge\tfrac{1}{r-r_0}\ge\tfrac1r.\]
For $r_0<y^j<r_0+1$ use the Markov property and the fact that there is a positive probability $\delta>0$ that the $\Ybar^j$ walk makes a jump of at least $\nu>0$ to the right to obtain
  \begin{align*}
  & \Pbar^j_{y^j}\{\Ybar^j_\zeta>r\}\\ 
  &\qquad\ge
  \Pbar^j_{y^j}\big\{\Ybar^j \text{ reaches } [r_0+1,r] \text{ in fewer than }\tfrac1{\nu} \text{ steps}\big\}
  \inf_{y\in [r_0+1,r]}\!\!\!\Pbar^j_y\{\Ybar^j_\zeta>r\}\\
   &\qquad\ge \tfrac{\delta^{1/\nu}}{r}.
     \end{align*}

Thus, the first term on the right-hand side of \eqref{eq4} is bounded by $1-\frac{C}{r}$. 
The proposition is proved if we choose $A_1$ large enough.
 \end{proof}  
 
The following consequences follow immediately.
 
 \begin{corollary}
 Fix a constant $c_1>1$ and consider positive integers $r_0$ and $r$ that satisfy 
 	\[\log\log r\le r_0\le c_1\log\log r<r.\]
Then for large enough $r$
	\[\inf_{x\in B_r\setminus B_{r_0}}\PMC_x\big\{\text{without entering $B_{r_0}$ chain $Y$ exits $B_r$ by time $r^4$}\big\}\ge r^{-3}.\]
 \end{corollary}

\begin{proof}
The idea is to apply the previous proposition recursively to go from scale $\log\log r$ to scale $r$.
The proof of Corollary A.1 of \cite{rass-sepp-09} goes through word for word with the choice $\gamma=\frac{p}{3}$.  
\end{proof}

 \begin{lemma}
Let $U_r=\inf\{n\ge0:Y_n\notin B_r\}$ be the first exit time from $B_r=[-r,r]^d$ for the Markov chain $Y$. Then there exists a finite positive constant $C$ such that 
	\[\sup_{x\in B_r} \EMC_x[U_r]\le C r^{13}\text{\quad for all }r>0.\]
 \end{lemma}
 
 \begin{proof}
The proof of Lemma A.4 of  \cite{rass-sepp-09} works in our setting as well since it only uses the above corollary, the fact that the exit times satisfy \eqref{exponential growth},
and  general Markov chain facts.  
 \end{proof}
 
 We now complete the proof of \eqref{eq2}. 
 Let $B=B_r$ with $r=n^{\varepsilon}$. Let $0=V_0^{\rm in}<V_1^{\rm out}<V_1^{\rm in}<V_2^{\rm out}<V_2^{\rm in}<\cdots$ be the successive entrance times $V_i^{\rm in}$ into $B$ and exit times $V_i^{\rm out}$ from $B$ for the Markov chain $Y$. 
%
Write
\begin{align}
\sum_{k=0}^{n-1} \PMC_0\{Y_k \in [-n^{\varepsilon},n^{\varepsilon}]^d\}
&\le\sum_{i=0}^n \EMC_0\big[(V^{\rm out}_{i+1}-V^{\rm in}_i)\one\{V^{\rm in}_i\le n\}\big]\nn\\
&\le\sup_{|y|\le n^\varepsilon} \EMC_y[U_{n^\e}]\cdot \EMC_0\Big[\sum_{i=0}^n\one\{V_i^{\rm in}\le n\}\Big]\nn\\
&\le Cn^{13\varepsilon}\,\EMC_0\Big[\sum_{i=0}^n\one\Big\{\sum_{j=1}^i(V_j^{\rm in}-V_j^{\rm out})\le n\Big\}\Big].\label{visits estimate}
\end{align}

 \begin{lemma} There exists a positive constant $C$ such that the following holds: excursion lengths $\{V_j^{\rm in}-V_j^{\rm out}: 1\le j\le n\}$ stochastically dominate i.i.d.\ random variables $\{\eta_j\}$ such that $1\le\eta_j\le n^{p\varepsilon}$, almost surely, and $P\{\eta_j \ge a\} \ge  Ca^{-1/2}$ for $1\le a\le n^{p\varepsilon}$.
 \end{lemma} 
 \begin{proof}
 Let $V$ denote the first entrance time into $B$. We will show that $\PMC_y\{V\ge a\}$  is bounded below uniformly over $y\notin B$. As in  Proposition \ref{a3}, let us assume that the $\Ybar$ walk is nondegenerate along direction $j$. Assume $y^j>r$, the other case being similar by symmetry. Let $w_r=\inf\{n\ge 1:\Ybar_n^j \le r\}$. Then
 \[ \Pbar_y\{V\ge a\} \ge \Pbar^j_{y^j}\{w_r\ge a\}\]
 and it follows from Theorem 1a of \cite[page 415]{fell-2} 
 that $\Pbar^j_{y^j}\{w_r\ge a\}\ge\frac{C}{\sqrt{a}}$.  By arguments similar to Proposition \ref{a3} (look at $X$ and $\Xtil$ paths up to time $a$), we get that for $y\notin B$,
 \begin{eqnarray*}
 \PMC_y\{V\ge a\} &\ge& \big(1-\tfrac{C}{n^{p\varepsilon}}\big)^a\Pbar_y\{V\ge a\}\\
 &\ge & \big(1-\tfrac{C}{n^{p\varepsilon}}\big)^a \tfrac{C}{\sqrt{a}} \ge\tfrac{C}{\sqrt{a}}
  \end{eqnarray*}
 for $1\le a\le n^{p\varepsilon}$ and some constant $C>0$. This implies the stochastic domination. 
 Assuming that $1\le\eta_j\le n^{p\varepsilon}$ only weakens the conclusion. 
\end{proof}

Let $K_n=\inf\{k:\sum_{j=1}^k\eta_j>n\}$ be the number of renewals up to time $n$. Wald's inequality gives $E[K_n]E[\eta_1]\le n+n^{p\varepsilon}\le2n$,
while the tail of $\eta_1$ gives $E[\eta_1]\ge Cn^{p\varepsilon/2}$. Consequently, $E[K_n]\le Cn^{1-p\varepsilon/2}$.
Picking up from \eqref{visits estimate} we have
\begin{align*}
\sum_{k=0}^{n-1} \PMC_0\{Y_k \in [-n^{\varepsilon},n^{\varepsilon}]^d\}
&\le Cn^{13\varepsilon}\,E\Big[\sum_{i=0}^n\one\Big\{\sum_{j=1}^i\eta_j\le n\Big\}\Big]\\
&\le C n^{13\varepsilon}E[K_n]\le C n^{1+13\varepsilon-p\varepsilon/2}.
\end{align*}
This proves \eqref{eq2} and the proof of Proposition \ref{the proposition} is then complete, with $\eta=\max(1-p\varepsilon,1-(p/2-13)\varepsilon)\le1/2+13/p<1$.
\end{proof}

\begin{proof}[Proof of Theorem \ref{thm}] Since the measure $\P_{\infty}$ we constructed satisfies the assumptions of Theorem \ref{thm2}, the conclusion of Theorem \ref{thm2} holds. But we already know that $\P\big\vert_{\kS_0^+}=\P_{\infty}\big\vert_{\kS_0^+}$ and since the walk always remains in the region $\{(n,x):n\ge 0\}$, the conclusion of Theorem 2 holds with $\P_{\infty}$ replaced by $\P$. Finally, since the diffusion matrix is nonrandom, the averaged invariance principle holds with the
same matrix and thus this matrix is no other than $\mathcal D$.
\end{proof}

\noindent\textbf{Acknowledgement.} We thank Timo Sepp\"al\"ainen for useful discussions.

\bibliographystyle{alea2}
\bibliography{qclt_cs_refs}

\end{document}